\documentclass[12pt]{amsart}
\usepackage{amsmath,amsxtra,amssymb,latexsym,hyperref,color}
\usepackage{xcolor}
\usepackage{array}
\usepackage{comment}
\usepackage{graphicx}

\newtheorem{theorem}{Theorem}[section] 
\newtheorem{corollary}[theorem]{Corollary}
\newtheorem{lemma}[theorem]{Lemma}
\newtheorem{proposition}[theorem]{Proposition}

\numberwithin{equation}{section}

\begin{document}

\title{Morse-Novikov cohomology on foliated manifolds}

\author[]{Md. Shariful Islam}

\address{Department of Mathematics \\
University of Dhaka\\
Dhaka 1000, Bangladesh}

\email[M. S.~Islam]{mdsharifulislam@du.ac.bd}



\thanks{*Corresponding author: mdsharifulislam@du.ac.bd}

\begin{abstract}
The idea of Lichnerowicz or Morse-Novikov cohomology groups of a manifold has been utilized by many researchers to study important properties and invariants of a manifold. Morse-Novikov cohomology is defined using the differential $d_\omega=d+\omega\wedge$,
where $\omega$ is a closed $1$-form. We study Morse-Novikov cohomology relative to a foliation on a manifold and its homotopy invariance and then extend it to more general type of forms on a Riemannian foliation. We study the Laplacian and Hodge decompositions for the corresponding differential operators on reduced leafwise Morse-Novikov complexes. In the case of Riemannian foliations, we prove that the reduced leafwise Morse-Novikov cohomology groups satisfy the Hodge theorem and Poincar{\'e} duality. The resulting isomorphisms yield a Hodge diamond structure for leafwise Morse-Novikov cohomology.

\end{abstract}

\maketitle

%\tableofcontents
\maketitle{}
\section{Introduction}
Consider an $n$-dimensional smooth manifold $M$; denote by $\Omega^{k}(M)$ the collection of all degree $k$ differential forms on $M$ and by $H^{k}(M)$ the corresponding  de Rham cohomology group. Let $\omega$ be a closed $1$-form that is not necessarily exact. We consider the twisted operator $d_{\omega}:\Omega^{k}(M)\rightarrow \Omega^{k+1}(M)$ defined by $d_{\omega}=d+\omega\wedge$, where $d$ is the usual exterior derivative, so that $(d_{\omega})^{2}=0$. The differential cochain complex $(\Omega^{*}(M),d_{\omega})$ is called the Morse-Novikov complex of the manifold $M$. The cohomology groups $H_{\omega}^{k}(M)$ of this cochain complex are called the Morse-Novikov or Lichnerowicz cohomology groups of $M$ and have been utilized by many researchers. Morse-Novikov cohomology was first studied by A. Lichnerowicz in \cite{lichnerowicz1977varietes}, and used in the context of Poisson geometry. The idea of Lichnerowicz has been exploited to study many properties of manifolds. In \cite {MR630459} and \cite{MR676612} S.P.Novikov proved a generalization of the Morse inequalities by comparing the ranks of these cohomology groups with combinatorial invariants derived from the zeros of the form $\omega$. Pazhintov \cite{pazhitnov1987analytic} gave an analytic proof of the real part of Novikov's inequalities. E. Witten used the Morse-Novikov cohomology for exact $\omega$ in his famous discovery \cite{witten1982supersymmetry} of what is now known as \textit{Witten deformation}. In this case Morse-Novikov cohomology is isomorphic to de Rham cohomology. M. Shubin and S. P. Novikov applied the deformation method to a rigorous treatment of eigenvalue limits of Witten Laplacians for more general $1$-forms and vector fields in \cite{MR856461} and \cite{MR1398925}. 
Many other researchers have extended and generalized this work, such as Braverman and Farber  \cite{MR1458239} in cases of nonisolated zeros of $1$-forms and vector fields. See \cite{MR2034601} for a good reference on these related topics. Alexandra Otiman studied Morse-Novikov cohomology for particular classes of closed $1$-forms in \cite{otiman2016morse}. I. Vaisman studied locally conformal symplectic manifolds in \cite{vaisman1980remarkable}, and  L. Ornea, and M. Verbitsky studied Morse-Novikov cohomology of locally conformally K{\"a}hler manifolds in  \cite{ornea2009morse}. In \cite{chen2019morse}, X. Chen showed that if a Riemannian manifold $M$ has almost non-negative sectional curvature and nonzero first de Rham cohomology group, then all the Morse-Novikov cohomology groups of $M$ vanish irrespective of the choice of the closed non-exact $1$-form $\omega$. In \cite{Meng2019}, L. Meng established an analogue of the Leray-Hirsch Theorem for de Rham cohomology and a blowup formula fo Dolbeault-Morse-Novikov cohomology on complex manifolds. Morse-Novikov cohomology theory has also been used to study locally conformal symplectic manifolds  (see \cite{vaisman1982generalized}, \cite{vaisman1980remarkable}, and \cite{vaisman1985locally}).

Two types of Morse-Novikov cohomology associated to foliations are basic (see \cite {MR107280}) and leafwise. Liviu Ornea and Vladimir Slesar studied basic Morse-Novikov cohomology in \cite{ornea2016basic}. K. Richardson and G. Habib used basic Morse-Novikov cohomology to prove that the basic signature and the \'Alvarez class of a Riemannian foliation are homotopy invariants \cite{habib2017homotopy}. They have also used a modified differential as in Morse-Novikov cohomology to define a twisted basic cohomology for Riemannian foliations that satisfies Poincar{\'e} duality \cite{habib2013modified}. J. A. \'Alvarez Lopez, Y. Kordyukov, and E. Leichtnam studied leafwise Hodge decomposition on Riemannian foliations with bounded geometry and extended the Morse-Novikov differential complex \cite{lopez2019analysis}. 

This paper is constructed as follows. Using $d+\omega\wedge$ as the differential for a leafwise closed $1$-form $\omega$, we study leafwise Morse-Novikov cohomology groups whose isomorphism classes turn out to be smooth invariants of the foliation. In the cases where $\omega$ is truly a closed $1$-form on the manifold, one further obtains Morse-Novikov cohomology groups from the foliation. In Section 3 we study the basic properties of leafwise Morse-Novikov cohomology groups, including the homotopy axiom in Proposition \ref{ta13}, and foliated homotopy invariance in Corollary \ref{ta14}. With the additional assumption that the foliation is Riemannian, we give a proof of the Hodge decomposition in Corollary \ref{ta20} and Poincar{\'e} duality in Corollary \ref{ta21}. We extend these results to more general settings of forms of $p,q$ type: homotopy axioms for general leafwise Morse-Novikov cohomology in Proposition \ref{ta24}, Hodge decomposition for general leafwise Morse-Novikov cohomology in Theorem \ref{ta25}, Poincar{\'e} duality for general leafwise Morse-Novikov cohomology in Proposition \ref{ta27}. The assumption that the foliation is Riemannian is required to obtain Hodge theory and Poincar{\'e} duality; for general smooth foliations, those results are false, even for the case when $\omega=0$. In the remainder of Section \ref{lastsection}, the isomorphisms between the Morse-Novikov cohomology groups are discussed, leading to Corollary \ref{hodgediamond}, that shows the Hodge diamond structure.

Much but not all of the results in this paper were part of the authors Ph.D thesis \cite{MR4060632}.

\section{Leafwise de~Rham cohomology}
In this section we review notations and known results. Let $M$ be a closed compact oriented Riemannian manifold. Suppose we are given a smooth foliation $\left(M,\mathcal{F}\right)$ on $M$. Let $T\mathcal{F}$ denote the leafwise tangent bundle and  $T^{*}\mathcal{F}$ denotes its dual bundle. Let the conormal bundle $N^{*}_{x}\mathcal{F}$ be defined at each point $x\in M$ as the set of all linear functionals that  map each vector in $T_{x}\mathcal{F}$ to zero, and this bundle can be canonically identified with a subbundle of $T^{*}_{x}M$ independent of the metric. Using the metric, the normal bundle $Q=TM\slash T\mathcal{F}$ may be uniquely identified with a subbundle $N\mathcal{F}=T\mathcal{F}^{\perp}$ of $TM$. Also, we can identify the dual bundle $T^{*}\mathcal{F}$ with the set of covectors that kill $N\mathcal{F}$. Now we can decompose all differential forms using
\begin{equation}
\label{one}
\Lambda^{u,v}(M,\mathcal{F})=\Lambda^{u} N^{*}\mathcal{F}\wedge\Lambda^{v} T^{*}\mathcal{F} 
\end{equation}
Let $\Omega^{u,v}(M,\mathcal{F})=\Gamma\Lambda^{u,v}(M,\mathcal{F})$. The exterior derivative can then be decomposed as $d=d_{0,1}+d_{1,0}+d_{2,-1}$ with
\[
d_{i,j}\omega\in \Omega^{u+i,v+j}(M,\mathcal{F})
\]
for all $\omega\in\Omega^{u,v}(M,\mathcal{F})$. Then it is easy to see that since
$d^2=0$, we also have $d_{0,1}^2=0$.

The elements of $\Omega^{0,k}\left(M,\mathcal{F}\right)$ are called leafwise $k$-forms. Let $\Gamma\left(T\mathcal{F}\right)$ be the set of smooth sections of $T\mathcal{F}$. If $X,Y\in\Gamma\left(T\mathcal{F}\right)$, then by the Frobenius theorem $[X,Y]\in\Gamma\left(T\mathcal{F}\right)$. The leafwise exterior differential operator $d_{0,1}^{k}=d_{\mathcal{F}}^{k}:\Omega^{0,k}\left(M,\mathcal{F}\right)\rightarrow\Omega^{0,k+1}\left(M,\mathcal{F}\right)$ may also be defined by
\begin{eqnarray*}
d_{\mathcal{F}}^{k}\omega\left(X_{0},\cdots,X_{k+1}\right)=\sum_{0\le i\le k}(-1)^{i}[X_{i}\omega\left(X_{0},\cdots,\hat{X_{i}},\cdots,X_{k}\right)]\\
+\sum_{0\le i<j\le k}(-1)^{i+j}\omega\left([X_{i},X_{j}],X_{0},\cdots,\hat{X_{i}},\cdots,\hat{X_{j}},\cdots, X_{k}\right)
\end{eqnarray*}
for $X_{0},\cdots,X_{k+1}\in\Gamma\left(T\mathcal{F}\right)$. The differential operator $d_{\mathcal{F}}$ is the restriction of the usual differential on differential forms on the leaves of $\mathcal{F}$. Similar to the usual exterior differential, the leafwise differential satisfies $d_{\mathcal{F}}^{k+1}\circ d_{\mathcal{F}}^{k}=0$. For $k\ge 0$, the $k^{\mathrm{th}}$ leafwise cohomology  group $H^{k}\left(M,d_{\mathcal{F}}\right)$ is the $k^{\mathrm{th}}$ cohomology group 
$$H^{k}\left(M,d_{\mathcal{F}}\right)=\frac{\ker d_{\mathcal{F}}^{k}}{\mathrm{im}\,d_{\mathcal{F}}^{k-1}}$$
of the cochain complex $\left(\Omega^{0,*}\left(M,\mathcal{F}\right),d_{\mathcal{F}}\right)$.

For the purpose of having Laplacian and Hodge decompositions, we need to consider reduced leafwise cohomology $$\bar{H}^{k}\left(M,d_{\mathcal{F}}\right)=\frac{\ker d_{\mathcal{F}}^{k}}{\overline{\mathrm{im}\,d_{\mathcal{F}}^{k-1}}}.$$ 
Here the closure $\overline{\mathrm{im}\, d_{\mathcal{F}}^{k-1}}$ is taken with respect to the Frech{\'e}t topology on $\Omega^{0,k}\left(M,\mathcal{F}\right)$.
The cup product induced from exterior product of forms makes $\bar{H}^{*}\left(M,d_{\mathcal{F}}\right)$ into a graded commutative algebra over $C^{\infty}(M)$.

Let $f:M\rightarrow N$ be a smooth map of the foliated manifold which maps leaves into leaves. Then the pullback maps 
$$f^{*}:\Gamma\left(\Lambda^k T^*\mathcal{F}_N\right)\rightarrow \Gamma\left(\Lambda^k T^*\mathcal{F}_M\right)$$ 
are defined for all $k$. They commute with $d_{\mathcal{F}}$ and respect the exterior product; therefore they induce a continuous map of the reduced cohomology ring. 
$$f^{*}:\bar{H}^{k}\left(N,d_{\mathcal{F}_{0}}\right)\rightarrow \bar{H}^{k}\left(M,d_{\mathcal{F}_{1}}\right).$$ 
Such maps are called foliated maps.
Two smooth foliated maps $f, g : \left(N, \mathcal{F}_{0}\right)\rightarrow\left(M, \mathcal{F}_{1}\right)$ between two
foliated manifolds are leafwise homotopic if there is a map $F : N \times [0, 1]\rightarrow M$ such that if $F_{t}$ denotes the restriction $F_{t} = F|_{N\times{t}}:N\rightarrow M$ for $t\in[0, 1]$, then $F|_{0} = f, F|_{1} = g$, $F_{t}\left(\mathcal{F}_{0}\right)\in\mathcal{F}_{1}$ for all $t\in[0, 1]$, and for every $x\in N$ the points $F_{t_{1}}(x), F_{t_{2}}(x)$ lie in the same leaf of $\mathcal{F}_{1}$ for all $t_{1}, t_{2} \in [0, 1]$. Thus, a leafwise homotopy consists of leaf-preserving maps. Denote the identity maps of $\left(N,\mathcal{F}_{0}\right)$, $\left(M,\mathcal{F}_{1}\right)$ by $l_{\mathcal{F}_{0}}$, $l_{\mathcal{F}_{1}}$ respectively. A leafwise map $f:\left(N,\mathcal{F}_{0}\right)\rightarrow\left(M,\mathcal{F}_{1}\right)$ is a leafwise homotopy equivalence if there exists a leafwise map $g:\left(M,\mathcal{F}_{1}\right)\rightarrow\left(N,\mathcal{F}_{0}\right)$ with $f\circ g$ is leafwise homotopic to $l_{\mathcal{F}_{0}}$ and $g\circ f$ is leafwise homotopic to $l_{\mathcal{F}_{1}}$.

\begin{proposition}
\label{ta4} 
\cite[Theorem I, 3.2]{hoster2001derived}) If the maps $f, g : \left(N, \mathcal{F}_{0}\right)\rightarrow\left(M, \mathcal{F}_{1}\right)$ between two
foliated manifolds are leafwise homotopic, then $f^{*}=g^{*}:H^{k}\left(M,d_{\mathcal{F}_{1}}\right)\rightarrow H^{k}\left(N,d_{\mathcal{F}_{0}}\right)$. That is, leafwise homotopic maps induce the same map on leafwise cohomology groups. 
\end{proposition}

\begin{corollary}
\label{ta5}
If a map $f : \left(N, \mathcal{F}_{0}\right)\rightarrow\left(M, \mathcal{F}_{1}\right)$ is a smooth foliated homotopy equivalence, then $f^{*}$ induces an isomorphism between $H^{k}\left(N, d_{\mathcal{F}_{0}}\right)$ and $H^{k}\left(M, d_{\mathcal{F}_{1}}\right)$.
\end{corollary}

\section{Leafwise Morse-Novikov Cohomology and Hodge Theory}

From now we assume that our foliation is Riemannian, characterized by the existence of a bundle-like metric $g$ such that a geodesic of the metric $g$ is orthogonal to all leaves that it meets whenever it is orthogonal to one of them. We assume that a bundle like metric has been chosen.   

Let $M$ be an oriented manifold endowed with a foliation $\mathcal{F}$ of dimension $p$. The graded Frech{\'e}t space $\Omega^{0,*}\left(M,\mathcal{F}\right)$ can be endowed with the natural metric 
$$\left(\alpha,\beta\right)=\int_{M}\left\langle\alpha,\beta\right\rangle_{\mathcal{F}}\mathrm{vol}.$$
In this formula $\left\langle,\right\rangle_{\mathcal{F}}$ is the Riemannian metric on $\Lambda T^{*}\mathcal{F}$ induced from the Riemannian metric $g$ on $M$, and $\mathrm{vol}$ is the volume form associated to the metric $g$. We denote the formal adjoint of the leafwise differential $d_{\mathcal{F}}$ with respect to this inner product by $\delta_{\mathcal{F}}$; then the corresponding Laplacian is 
$$\Delta_{\mathcal{F}}=d_{\mathcal{F}}\delta_{\mathcal{F}}+\delta_{\mathcal{F}}d_{\mathcal{F}}.$$
Since $\mathcal{F}$ is Riemannian, the restriction of $\delta_{\mathcal{F}}$ to any leaf is the codifferential of the leaf with respect to the induced metric \cite[Lemma 3.2]{lopez2001long}, i.e. $$\left(\delta_{\mathcal{F}}\alpha\right)|_{F}=\delta_{\mathcal{F}}\left(\alpha\right)|_{F}\text{ for all }\alpha\in\Omega^{0,k}\left(M,\mathcal{F}\right),$$
where $F$ denotes a leaf of the foliation. Now we assume that the tangent bundle $T\mathcal{F}$ is orientable. The choice of an orientation determines a volume form $\chi_{\mathcal{F}}\in\Omega^{0,p}\left(M,\mathcal{F}\right)$. Now we can define leafwise Hodge star-operator
 $$*_{\mathcal{F}}:\Lambda^{0,k}T^{*}\mathcal{F}\rightarrow\Lambda^{0,p-k}T^{*}\mathcal{F}\text{ for each } k\text{ and }x\in M,$$ 
and it is determined by the relation 
$$\alpha\wedge*_{\mathcal{F}}\beta=\left\langle\alpha,\beta\right\rangle\chi_{\mathcal{F}},\text{ for }\alpha,\beta\in\Lambda^{0,k}T^{*}_{x}\mathcal{F}.$$
 This fibrewise star-operator determines the leafwise star-operator
 $$*_{\mathcal{F}}:\Omega^{0,k}\left(M,\mathcal{F}\right)\rightarrow\Omega^{0,p-k}\left(M,\mathcal{F}\right)\text{ for each } k.$$ 

Now we state some important properties of leafwise cohomology. Suppose $M$ is compact, and $\mathcal{F}$ is a $p$-dimensional oriented Riemannian foliation of $M$ with a bundle-like metric $g$.

\begin{proposition}
\label{ta7}
\cite[Theorem 0.2]{deninger2001counterexample} The map $\phi:\ker\Delta^{k}_{\mathcal{F}}\rightarrow\bar{H}^{k}\left(M,d_{\mathcal{F}}\right)$ defined by $\phi\left(\omega\right)=\omega\,\mathrm{mod}\,\overline{\mathrm{im}\,d^{k}_{\mathcal{F}}}$ is a topological isomorphism of Frech{\'e}t spaces. This isomorphism, in general, does not hold for non-Riemannian foliations.
\end{proposition}
Under the same assumptions, the next deep result is due to {\'A}lvarez L{\'o}pez and Kordyukov.
\begin{theorem}
\label{ta8}
\cite[Corollary C]{lopez2001long} The Hodge star-operator induces an isomorphism 
$$*_{\mathcal{F}}:\ker\Delta^{k}_{\mathcal{F}}\rightarrow\ker\Delta^{p-k}_{\mathcal{F}}.$$
Moreover $*_{\mathcal{F}}$ commutes with $\Delta^{k}_{\mathcal{F}}$ up to a sign. From the previous proposition we have the following isomorphism $$*_{\mathcal{F}}:\bar{H}^{k}\left(M,d_{\mathcal{F}}\right)\rightarrow\bar{H}^{p-k}\left(M,d_{\mathcal{F}}\right).$$
\end{theorem}

Let $\omega$ be a leafwise closed $1$-form, which is not necessarily leafwise exact. We consider the twisted operator $d^{\omega}_{\mathcal{F}}:\Omega^{0,k}\left(M,\mathcal{F}\right)\rightarrow \Omega^{0,k+1}\left(M,\mathcal{F}\right)$ defined by $d^{\omega}_{\mathcal{F}}=d_{\mathcal{F}}+\omega\wedge$, where $d_{\mathcal{F}}$ is the exterior derivative along the leaf. Since $d_{\mathcal{F}}\circ d_{\mathcal{F}}=d^{2}_{\mathcal{F}}=0$, $\omega\wedge\omega=0$, and $d_{\mathcal{F}}(\omega\wedge\alpha)=-\omega\wedge d_{\mathcal{F}}\alpha$ for any $k$-form $\alpha$, it follows that $\left(d^{\omega}_{\mathcal{F}}\right)^{2}=0$. The differential cochain complex $\left(\Omega^{0,*}\left(M,\mathcal{F}\right),d^{\omega}_{\mathcal{F}}\right)$ is called the leafwise Morse-Novikov complex of the foliated manifold $\left(M,\mathcal{F}\right)$. Let $d_{\mathcal{F}}^{\omega,k}$ be the restriction of $d^{\omega}_{\mathcal{F}}$ to $\Omega^{0,k}\left(M,\mathcal{F}\right)$. The cohomology groups $$H^{k}_{\omega}\left(M,d_{\mathcal{F}}\right)=\frac{\ker\left(d_{\mathcal{F}}^{\omega,k}\right)}{\mathrm{im}\left(d_{\mathcal{F}}^{\omega,k-1}\right)}$$ of this cochain complex are called the leafwise Morse-Novikov cohomology groups of $\left(M,\mathcal{F}\right)$. For the purpose of obtaining Hodge decomposition, we need to consider the reduced leafwise Morse-Novikov cohomology $$\bar{H}^{k}_{\omega}\left(M,d_{\mathcal{F}}\right)=\frac{\ker\left(d_{\mathcal{F}}^{\omega,k}\right)}{\overline{\mathrm{im}\left(d_{\mathcal{F}}^{\omega,k-1}\right)}}.$$
Here the closure $\overline{\mathrm{im} \left(d_{\mathcal{F}}^{\omega,k-1}\right)}$ is taken with respect to the Frech{\'e}t topology on $\Omega^{0,k}\left(M,\mathcal{F}\right)$.
The cup product induced from exterior product of forms makes $\bar{H}^{*}_{\omega}\left(M,d_{\mathcal{F}}\right)$ into a graded commutative algebra over $C^{\infty}(M)$.

\begin{proposition}
\label{ta9}
If $\omega$ and $\theta=\omega+d_{\mathcal{F}}g$ are cohomologous in $H^{1}\left(M,d_{\mathcal{F}}\right)$, then for each $k$, the leafwise Morse-Novikov cohomology groups $H_{\omega}^{k}\left(M,d_{\mathcal{F}}\right)$ and $H_{\theta}^{k}\left(M,d_{\mathcal{F}}\right)$ are isomorphic. That is, the map  $\Phi:H_{\omega}^{k}\left(M,d_{\mathcal{F}}\right)\rightarrow H_{\theta}^{k}\left(M,d_{\mathcal{F}}\right)$ given by $\Phi\left([\alpha]\right)=[e^{-g}\alpha]$ is an isomorphism.
\end{proposition}

\begin{proof}
If $\omega$ and $\theta$ are cohomologous, then there exists $g\in \Omega^{0}(M)$, such that $\theta-\omega=d_{\mathcal{F}}g$. Define the mapping $\phi : H_{\omega}^{k}(M,d_{\mathcal{F}})\rightarrow H_{\theta}^{k}(M,d_{\mathcal{F}})$ by $\phi(\left[\alpha\right])=\left[e^{-g}\alpha\right]$. One can check that $\phi$ is well-defined and is a group homomorphism, since
\begin{eqnarray*}
\left(d_{\mathcal{F}}+\theta\wedge\right)\left(e^{-g}\alpha\right)&=&\left(d_{\mathcal{F}}+\omega\wedge+d_{\mathcal{F}}g\wedge\right)\left(e^{-g}\alpha\right),\\
&=&e^{-g}\left(d_{\mathcal{F}}+\omega\wedge\right)\left(\alpha\right),
\end{eqnarray*}
for all $\alpha\in\Omega^{k}\left(M,d_{\mathcal{F}}\right)$.

Suppose $\alpha,\beta\in\Omega^{k}\left(M,d_{\mathcal{F}}\right)$ are cohomologous, then there exists $\nu\in\Omega^{k-1}\left(M,d_{\mathcal{F}}\right)$ such that $\alpha-\beta=\left(d_{\mathcal{F}}+\omega\wedge\right)\nu$. We have 
\begin{eqnarray*}
\phi\left([\alpha-\beta]\right)&=&[e^{-g}\left(d_{\mathcal{F}}+\omega\wedge\right)\nu],\\
\Rightarrow\left[\left(d_{\mathcal{F}}+\theta\wedge\right)\left(e^{-g}\nu\right)\right]&=&\left[0\right].
\end{eqnarray*}

Similarly if, $\phi(\left[\alpha\right])=0$, Then $\left[\alpha\right]=0 \in H_{\omega}^{k}(M,d_{\mathcal{F}})$, and $\phi$ is injective.\newline
 If $\left[\alpha\right]\in H_{\theta}^{k}(M,d_{\mathcal{F}})$ then we find similarly that $\left[e^{g}\alpha\right]\in H_{\omega}^{k}(M,d_{\mathcal{F}})$, so that $\phi$ is surjective.
\end{proof}

\begin{corollary}
\label{ta10}
If $\omega$ is a $d_{\mathcal{F}}$ exact $1$-form, then for each $k$ the leafwise Morse-Novikov cohomology group $H_{\omega}^{k}\left(M,d_{\mathcal{F}}\right)$ and the leafwise de~Rham cohomology group $H^{k}\left(M,d_{\mathcal{F}}\right)$ are isomorphic. $$H_{\omega}^{k}\left(M,d_{\mathcal{F}}\right)\cong H^{k}\left(M,d_{\mathcal{F}}\right).$$ 
\end{corollary}

\begin{corollary}
\label{ta11}
If the first leafwise de~Rham cohomology group $H^{1}\left(M,d_{\mathcal{F}}\right)$ equals $0$, then for every $d_{\mathcal{F}}$ closed 1-form $\omega$ and for each $k$ the leafwise Morse-Novikov cohomology groups satisfy $H_{\omega}^{k}\left(M,d_{\mathcal{F}}\right)=H^{k}\left(M,d_{\mathcal{F}}\right)$.
\end{corollary}

\begin{lemma}
\label{ta12}
For any smooth foliation $\left(M,\mathcal{F}\right)$ the leafwise Morse-Novikov cohomology $H_{\omega}^{0}\left(M,d_{\mathcal{F}}\right)=\{0\}$ if and only if $\omega$ is not $d_{\mathcal{F}}$ exact. 
\end{lemma}

\begin{proof}
 Suppose first that $H_{\omega}^{0}\left(M,d_{\mathcal{F}}\right)\ne\{0\}$ for a $d_{\mathcal{F}}$ closed one form $\omega$ on $\left(M,\mathcal{F}\right)$, then there is a nonzero function $f\in C^{\infty}\left(M\right)$, such that 
\begin{eqnarray*}
\left(d_{\mathcal{F}}+\omega\right)f &=& 0\\
d_{\mathcal{F}}f+f\omega &=& 0\\
d_{\mathcal{F}}\left(\log(\frac{1}{f})\right) &=& \omega,
\end{eqnarray*}
which implies $\omega$ is $d_{\mathcal{F}}$ exact.
Conversely, suppose that $\omega$ is $d_{\mathcal{F}}$ exact. There exists a function $g\in C^{\infty}\left(M\right)$ such that $d_{\mathcal{F}}g=\omega$.  Then
\begin{eqnarray*}
\left(d_{\mathcal{F}}+\omega\right)\left(e^{-g}\right) &=&-e^{-g} d_{\mathcal{F}}g+e^{-g}d_{\mathcal{F}}g=0,
\end{eqnarray*}
which shows $H_{\omega}^{0}\left(M,d_{\mathcal{F}}\right)\ne\{0\}$.
\end{proof}

\begin{proposition}
\label{ta13}
(Homotopy axiom for the leafwise Morse-Novikov cohomology). Let $f:\left(M,\mathcal{F}_{0}\right)\rightarrow \left(N,\mathcal{F}_{1}\right)$ and $g:\left(M,\mathcal{F}_{0}\right)\rightarrow \left(N,\mathcal{F}_{1}\right)$ be foliated homotopic maps, and let $\omega$ be a leafwise closed $1-\text{form}$ on $\left(N,\mathcal{F}_{1}\right)$. Then there exists a positive function $h:\left(M,\mathcal{F}_{0}\right)\rightarrow\mathbb{R}$ such that for all $k$ $$f^{*}=h g^{*}:H^{k}_{\omega}\left(N,d_{\mathcal{F}_{1}}\right)\rightarrow H^{k}_{f^{*}\omega}\left(M,d_{\mathcal{F}_{0}}\right).$$
\end{proposition}

\begin{proof}
Since $f$ and $g$ are foliated homotopic maps, by the homotopy axiom of leafwise de~Rham cohomology (Proposition $\ref{ta4}$), they induce the same map in leafwise de~Rham cohomology. Therefore, for any leafwise closed $1-$form $\omega\in\Omega^{1}\left(N,\mathcal{F}_{1}\right)$,  the pullback forms $f^{*}\omega\text{ , }g^{*}\omega\in H^{1}_{\omega}\left(M,d_{\mathcal{F}_{0}}\right)$ are cohomologous. There exists a function $\nu:\left(M,\mathcal{F}_{0}\right)\rightarrow\mathbb{R}$ such that $g^{*}\omega-f^{*}\omega=d_{\mathcal{F}}\nu$. We define $h=e^{\nu}$. Then from the proof of Proposition $\ref{ta9}$, for any $d^{\omega}_{\mathcal{F}}$ closed form $\alpha$ on $\left(N,\mathcal{F}_{1}\right)$, $\left[hg^{*}\alpha\right]=\left[f^{*}\alpha\right]\in H_{f^{*}\omega}^{k}\left(M,d_{\mathcal{F}_{0}}\right)$.
\end{proof}

\begin{corollary}
\label{ta14}
If $f:\left(M,\mathcal{F}_{0}\right)\rightarrow \left(N,\mathcal{F}_{1}\right)$ is a foliated homotopy equivalence and $\omega$ is a leafwise closed $1-\text{form}$, then the leafwise Morse-Novikov cohomology groups $H_{\omega}^{*}\left(M,d_{\mathcal{F}_{0}}\right)$ and $H_{f^{*}\omega}^{*}\left(N,d_{\mathcal{F}_{1}}\right)$ are isomorphic; i.e. $$H_{\omega}^{k}\left(M,d_{\mathcal{F}_{0}}\right)\cong H_{f^{*}\omega}^{k}\left(N,d_{\mathcal{F}_{1}}\right), \text{ for all } k.$$
\end{corollary}

\begin{proof}
There exists a map $g:\left(N,\mathcal{F}_{1}\right)\rightarrow \left(M,\mathcal{F}_{0}\right)$ such that $f\circ g$ is homotopic to the identity map $\mathbb{I}_{N}$ of $N$ and $g\circ f$ is homotopic to the identity map $\mathbb{I}_{M}$ of $M$ . We have linear maps
 $$H^{*}_{\omega}\left(N,\mathcal{F}_{1}\right)\stackrel{f^{*}}{\rightarrow}H^{*}_{f^{*}\omega}\left(M,\mathcal{F}_{0}\right)\stackrel{g^{*}}{\rightarrow}H^{*}_{g^{*}f^{*}\omega}\left(N,\mathcal{F}_{1}\right).$$
 By the homotopy axiom of leafwise Morse-Novikov cohomology (Proposition \ref{ta13}), there exists a positive function $h:N\rightarrow\mathbb{R}$ such that $g^{*}f^{*}\omega=\omega+d_{\mathcal{F}_{1}}(\log(h))$ then we have
$$\mathbb{I}_{N}=hg^{*}f^{*}=h\left(f\circ g\right)^{*}:H^{*}_{\omega}\left(N,\mathcal{F}_{1}\right)\rightarrow H^{*}_{\omega}\left(N,\mathcal{F}_{1}\right).$$ 
And similarly, for some positive function $\bar{h}:M\rightarrow\mathbb{R}$ such that $f^{*}g^{*}\omega=\omega+d_{\mathcal{F}_{0}}(\log(\bar{h}))$ then  we have
$$\mathbb{I}_{M}=\bar{h}f^{*}g^{*}=h\left(g\circ f\right)^{*}:H^{*}_{\omega}\left(M,\mathcal{F}_{0}\right)\rightarrow H^{*}_{\omega}\left(M,\mathcal{F}_{0}\right).$$ 
Since multiplication by a positive function is an isomorphism of Leafwise Morse-Novikov cohomology, $f^{*}$ and $g^{*}$ are isomorphisms.
\end{proof}

\section{Laplacian and Hodge decomposition on leafwise Morse-Novikov cohomology}
In the following, assume $\dim(M)=n$, $\dim(\mathcal{F})=p$, and $\mathrm{codim}(\mathcal{F})=q$. As in (Equation \eqref{one}), we have the bigrading $$\Omega^{u,v}\left(M,\mathcal{F}\right)=\Gamma(M,\Lambda^{v}T\mathcal{F}^{*}\otimes\Lambda^{u}T\mathcal{F}^{\perp *}).$$
We choose a tangential and a transversal orientation for $\mathcal{F}$ on any open subset $\mathcal{U}\subset M$. We obtain the Hodge star operator $*_{\mathcal{F}}$ on $T\mathcal{F}^{*}$ and $*_{\bot}$ on $T\mathcal{F}^{\bot*}$ to $\mathcal{U}$ such that $*_{\bot}(1)\wedge*_{\mathcal{F}}(1)$ is a positive volume form on $\mathcal{U}\subset M$.

\begin{lemma}
\label{ta15}
(Lemma 3.2 in \cite{lopez2001long}) The Hodge star operator on $\wedge\left(T^{*}M\right)=\wedge\left(T\mathcal{F}^{\bot*}\right)\otimes \wedge\left(T\mathcal{F}^{*}\right)$ on $\mathcal{U}$ satisfies
$$*=(-1)^{(q-u)v}*_{\bot}\otimes *_{\mathcal{F}}:\Lambda^{u}T\mathcal{F}^{\bot*}\otimes\Lambda^{v}T\mathcal{F}^{*}\rightarrow \Lambda^{q-u}T\mathcal{F}^{\bot*}\otimes\Lambda^{p-v}T\mathcal{F}^{*}.$$
\end{lemma}

\begin{lemma}
\label{ta16}
$*^{2}=(-1)^{(u+v)(p+q+1)}$ on $\Lambda^{u}T\mathcal{F}^{\bot*}\otimes\Lambda^{v}T\mathcal{F}^{*}$.
\end{lemma}
\begin{proof}
Restricted to $\Lambda^{u}T\mathcal{F}^{\bot*}$ and $\Lambda^{v}T\mathcal{F}^{*}$, we have $*_{\bot}^{2}=(-1)^{u(q+1)}$ and $*_{\mathcal{F}}^{2}=(-1)^{v(p+1)}$.
\begin{eqnarray*}
*^{2}&=&(-1)^{(q-u)v}(-1)^{u(p-v)}*_{\bot}^{2}\otimes *_{\mathcal{F}}^{2}\\ 
&=&(-1)^{(q-u)v}(-1)^{u(p-v)}(-1)^{u(q+1)}(-1)^{v(p+1)}id_{\bot}\otimes id_{\mathcal{F}}\\
&=&(-1)^{(u+v)(p+q+1)}id_{\bot}\otimes id_{\mathcal{F}}.
\end{eqnarray*}
\end{proof}
 
\begin{lemma}
\label{ta17}
(Formula 17 in \cite{lopez2001long}) The adjoint $\delta_{\mathcal{F}}=d_{\mathcal{F}}^{*}$ of $d_{\mathcal{F}}$ is given by
$$\delta_{\mathcal{F}}\beta=d_{\mathcal{F}}^{*}\beta=(-1)^{pk+p+1}*_{\mathcal{F}}d_{\mathcal{F}}*_{\mathcal{F}}\beta,$$ for any $\beta\in\Omega^{0,k}\left(M,\mathcal{F}\right)$.
\end{lemma}
\begin{proof}
The standard proof that $d^{*}=\left(-1\right)^{nk+n+1}*d*$ on $n$-manifolds applies on a foliated manifold in a local neighborhood. 
\end{proof}

\begin{lemma}
\label{18}
$\omega\lrcorner=(-1)^{pk+p}*_{\mathcal{F}}\left(\omega\wedge\right)*_{\mathcal{F}}$ for all $\omega\in\Omega^{0,k}\left(M,\mathcal{F}\right)$.
\end{lemma}

\begin{proof}
Suppose $\tau$ denotes the tangential volume form. For any $\beta\in\Omega^{0,k}$, we have
\begin{eqnarray*}
\omega\lrcorner\beta&=&\left(-1\right)^{n\left(k+1\right)}*\left(\omega\wedge\right)*\beta\\
&=&\left(-1\right)^{n\left(k+1\right)}*\left(\omega\wedge\right)\left(*_{\bot}\otimes*_{\mathcal{F}}\right)\left(1\otimes\beta\right)\\
&=&\left(-1\right)^{n\left(k+1\right)}\left(-1\right)^{qk}\left(-1\right)^{q}*\left(\tau\wedge\omega\wedge*_{\mathcal{F}}\beta\right)\\
&=&\left(-1\right)^{n\left(k+1\right)}\left(-1\right)^{qk}\left(-1\right)^{q}\left(*_{\bot}\otimes*_{\mathcal{F}}\right)\left(\tau\wedge\omega\wedge*_{\mathcal{F}}\beta\right)\\
&=&\left(-1\right)^{p\left(k+1\right)}*_{\mathcal{F}}\left(\omega\wedge\right)*_{\mathcal{F}}\beta.
\end{eqnarray*}
\end{proof}

By Lemmas above and the identity $*_{\mathcal{F}}^{2}=\left(-1\right)^{k\left(p-k\right)}$, we have on $\Omega^{0,k}\left(M,\mathcal{F}\right)$
\begin{eqnarray*}
\left(\omega\lrcorner\right)*_{\mathcal{F}}&=&\left(-1\right)^{k}*_{\mathcal{F}}\left(\omega\wedge\right)\\
*_{\mathcal{F}}\left(\omega\lrcorner\right)&=&\left(-1\right)^{k+1}\left(\omega\wedge\right)*_{\mathcal{F}}\\
*_{\mathcal{F}}d_{\mathcal{F}}^{*}&=&\left(-1\right)^{k}d_{\mathcal{F}}*_{\mathcal{F}}\\
d_{\mathcal{F}}^{*}*_{\mathcal{F}}&=&\left(-1\right)^{k+1}*_{\mathcal{F}}d_{\mathcal{F}}.
\end{eqnarray*}
The adjoint of the leafwise differential $\left(d_{\mathcal{F}}+\omega\wedge\right)$ is $\left(d_{\mathcal{F}}^{*}+\omega\lrcorner\right)$.  We denote the Laplacian corresponding to the differential $d_{\mathcal{F}}+\omega\wedge$ by $\Delta_{\mathcal{F},\omega}$. Then
$$\Delta_{\mathcal{F},\omega}=\left(d_{\mathcal{F}}+\omega\wedge\right)\left(d_{\mathcal{F}}^{*}+\omega\lrcorner\right)+\left(d_{\mathcal{F}}^{*}+\omega\lrcorner\right)\left(d_{\mathcal{F}}+\omega\wedge\right).$$

\begin{proposition}
\label{ta19}
If $\omega\in\Omega^{0,1}\left(M,\mathcal{F}\right)$ is a leafwise closed $1$-form, then the Hodge star operator $*_{\mathcal{F}}$ satisfies   $$*_{\mathcal{F}}\Delta_{\mathcal{F},\omega}=\Delta_{\mathcal{F},-\omega}*_{\mathcal{F}}.$$
\end{proposition}

\begin{proof}
 For all $\beta\in\Omega^{0,k}\left(M,\mathcal{F}\right)$, we have
\begin{eqnarray*}
*_{\mathcal{F}}\Delta_{\mathcal{F},\omega}\beta = *_{\mathcal{F}}(d_{\mathcal{F}}+\omega\wedge)(d^{*}_{\mathcal{F}}+\omega\lrcorner)\beta+*_{\mathcal{F}}(d^{*}_{\mathcal{F}}+\omega\lrcorner)(d_{\mathcal{F}}+\omega\wedge)\beta\\
=\left(-1\right)^{k}(d^{*}_{\mathcal{F}}-\omega\lrcorner)*_{\mathcal{F}}(d^{*}_{\mathcal{F}}+\omega\lrcorner)\beta+\left(-1\right)^{k+1}(d_{\mathcal{F}}-\omega\wedge)*_{\mathcal{F}}(d_{\mathcal{F}}+\omega\wedge)\beta\\
= \left(-1\right)^{k}\left(-1\right)^{k}(d^{*}_{\mathcal{F}}-\omega\lrcorner)(d_{\mathcal{F}}-\omega\wedge)*_{\mathcal{F}}\beta+\left(-1\right)^{k+1}\left(-1\right)^{k+1}(d_{\mathcal{F}}-\omega\wedge)(d^{*}_{\mathcal{F}}-\omega\lrcorner)*_{\mathcal{F}}\beta\\
= \left((d^{*}_{\mathcal{F}}-\omega\lrcorner)(d_{\mathcal{F}}-\omega\wedge)+(d_{\mathcal{F}}-\omega\wedge)(d^{*}_{\mathcal{F}}-\omega\lrcorner)\right)*_{\mathcal{F}}\beta\\
=\Delta_{\mathcal{F},-\omega}*_{\mathcal{F}}\beta.
\end{eqnarray*} 
Thus the operator $*_{\mathcal{F}}$ maps $\Delta_{\mathcal{F},\omega}$-harmonic forms to $\Delta_{\mathcal{F},-\omega}$-harmonic forms.
\end{proof}

\begin{corollary}
\label{ta20}
If we restrict the Laplacian $\Delta_{\mathcal{F},\omega}$ on $\Omega^{0,v}\left(M,\mathcal{F}\right)$, then $\ker\Delta_{\mathcal{F},\omega}$ is finite dimensional, and every reduced leafwise Morse-Novikov cohomology class has a $\Delta_{\mathcal{F},\omega}$ harmonic representative.
\end{corollary}

\begin{proof}
Notice that the operator $\Delta_{\mathcal{F},\omega}$ is defined on all forms in $\Omega^{u,v}\left(M,\mathcal{F}\right)$, but it is elliptic when restricted on the forms $\Omega^{0,v}\left(M,\mathcal{F}\right)$ along the leaves of the foliation (Section 1 in \cite{lopez2001long}). Using this ellipticity and the arguments similar to \cite{lopez2001long},    we can conclude that $\mathcal{H}^{k}_{\omega}\left(M,d_{\mathcal{F}}\right)=\ker\Delta_{\mathcal{F},\omega}\subset \Omega^{0,k}\left(M,\mathcal{F}\right)$ is isomorphic to $\bar{H}_{\omega}^{k}\left(M,d_{\mathcal{F}}\right)$, and  $$\mathcal{H}^{k}_{\omega}\left(M,d_{\mathcal{F}}\right)\cong \bar{H}_{\omega}^{k}\left(M,d_{\mathcal{F}}\right).$$
\end{proof}

\begin{corollary}
\label{ta21}
$\bar{H}^{k}_{\omega}\left(M,d_{\mathcal{F}}\right)\cong \bar{H}_{-\omega}^{p-k}\left(M,d_{\mathcal{F}}\right)$.
\end{corollary}

\begin{proof}
Since the operator $*_{\mathcal{F}}$ maps $\Delta_{\mathcal{F},\omega}$-harmonic forms to $\Delta_{\mathcal{F},-\omega}$-harmonic forms, it induces the isomorphism $$\mathcal{H}^{k}_{\omega}\left(M,d_{\mathcal{F}}\right)\cong \mathcal{H}_{-\omega}^{p-k}\left(M,d_{\mathcal{F}}\right).$$
\end{proof}

\section{Extension of leafwise Morse-Novikov cohomology to forms of general $u,v$ type}
We now extend leafwise Morse-Novikov cohomology to forms of general $u,v$ type.
Let $\omega\in\Omega^{0,1}\left(M,\mathcal{F}\right)$ be a leafwise closed $1$-form which is not necessarily exact. We consider the twisted operator $d^{\omega}_{\mathcal{F}}:\Omega^{u,v}\left(M,\mathcal{F}\right)\rightarrow \Omega^{u,v+1}\left(M,\mathcal{F}\right)$ defined by $d^{\omega}_{\mathcal{F}}=d_{\mathcal{F}}+\omega\wedge$, where $d_{\mathcal{F}}$ is the exterior derivative along the leaf. 
\begin{proposition}
$\left(d_{\mathcal{F}}+\omega\wedge\right)^{2}=0.$ Therefore $d_{\mathcal{F}}+\omega\wedge$ is a differential of the sections of $\Omega^{u,v}\left(M,\mathcal{F}\right)$.
\end{proposition}
\begin{proof}
Observe that for any section $\alpha\wedge\beta\in\Omega^{u,v}\left(M,\mathcal{F}\right)$, we have
\begin{eqnarray*}
\left(d_{\mathcal{F}}+\omega\wedge\right)^{2}\left(\alpha\wedge\beta\right)=\left(d_{\mathcal{F}}+\omega\wedge\right)\left((-1)^{u}\alpha\wedge d_{\mathcal{F}}\beta+\omega\wedge\left(\alpha\wedge\beta\right)\right)\\
=d_{\mathcal{F}}\left((-1)^{u}\alpha\wedge d_{\mathcal{F}}\beta\right)+(-1)^{u+1}\omega\wedge\left(\alpha\wedge d_{\mathcal{F}}\beta\right)+(-1)^{u}\left(\omega\wedge\left(\alpha\wedge d_{\mathcal{F}}\beta\right)\right)+\omega\wedge\left(\omega\wedge\left(\alpha\wedge\beta\right)\right)\\
=(-1)^{u+1}\left(\omega\wedge\left(\alpha\wedge d_{\mathcal{F}}\beta\right)-\omega\wedge\left(\alpha\wedge d_{\mathcal{F}}\beta\right)\right)=0.
\end{eqnarray*}
\end{proof}

We call the differential cochain complex $\left(\Omega^{*,*}\left(M,\mathcal{F}\right),d^{\omega}_{\mathcal{F}}\right)$ the general leafwise Morse-Novikov complex of the foliated manifold $\left(M,\mathcal{F}\right)$. The cohomology groups $$H^{*,*}_{\omega}\left(M,d_{\mathcal{F}}\right)=\frac{\ker\left(d_{\mathcal{F}}^{\omega}\right)}{\mathrm{im}\left(d_{\mathcal{F}}^{\omega}\right)}$$ of this cochain complex are called the general leafwise Morse-Novikov cohomology groups of $\left(M,\mathcal{F}\right)$. For the purpose of having Laplacian and Hodge decomposition, we need to consider reduced general leafwise Morse-Novikov cohomology $$\bar{H}^{*,*}_{\omega}\left(M,d_{\mathcal{F}}\right)=\frac{\ker\left(d_{\mathcal{F}}^{\omega}\right)}{\overline{\mathrm{im}\left(d_{\mathcal{F}}^{\omega}\right)}}.$$
Again the closure $\overline{\mathrm{im} \left(d_{\mathcal{F}}^{\omega}\right)}$ is taken with respect to the Frech{\'e}t topology on $\Omega^{*,*}\left(M,\mathcal{F}\right)$.

\begin{proposition}
\label{ta23}
If $\omega$ and $\theta=\omega+d_{\mathcal{F}}g$ are cohomologous in $H^{1}\left(M,d_{\mathcal{F}}\right)$, then for each $\ell,k$, the general leafwise Morse-Novikov cohomology groups $H_{\omega}^{\ell,k}\left(M,d_{\mathcal{F}}\right)$ and $H_{\theta}^{\ell,k}\left(M,d_{\mathcal{F}}\right)$ are isomorphic via the isomorphism $\left[\alpha\right]\mapsto \left[e^{-g}\alpha\right]$. 
\end{proposition}
\begin{proof}
Similar to the proof of Proposition \ref{ta9}.
\end{proof}

\begin{proposition}
\label{ta24}
(Homotopy axiom for the general leafwise Morse-Novikov cohomology). Let $f:\left(M,\mathcal{F}_{0}\right)\rightarrow \left(N,\mathcal{F}_{1}\right)$ and $g:\left(M,\mathcal{F}_{0}\right)\rightarrow \left(N,\mathcal{F}_{1}\right)$ be foliated homotopic maps, and $\omega\in\Omega^{0,1}\left(N,\mathcal{F}_{1}\right)$ be a leafwise closed $1-\text{form}$ on $\left(N,\mathcal{F}_{1}\right)$. Then there exists a positive function $h:\left(M,\mathcal{F}_{0}\right)\rightarrow\mathbb{R}$ such that, for all $\ell,k$ $$f^{*}=h g^{*}:H^{\ell,k}_{\omega}\left(N,d_{\mathcal{F}_{1}}\right)\rightarrow H^{\ell,k}_{f^{*}\omega}\left(M,d_{\mathcal{F}_{0}}\right).$$
\end{proposition}
\begin{proof}
Similar to the proof of Proposition \ref{ta13}.
\end{proof}

\section{Leafwise Morse-Novikov Hodge Theory of forms of general $\left(u,v\right)$ type}
\label{lastsection}
Let $P=d_{\mathcal{F}}+\omega\wedge$, then its formal adjoint $P^{*}$ is $\delta_{\mathcal{F}}-\omega\lrcorner$. Let $D=P+P^{*}$ be the corresponding Dirac operator. Then the corresponding Laplacian $\Delta_{\mathcal{F},\omega}^{u,v}=\left(P+P^{*}\right)^{2}=PP^{*}+P^{*}P$ is a nonnegative, self-adjoint second order differential operator on the smooth sections on $\Omega^{u,v}\left(M,\mathcal{F}\right)$. For each integer $k\geq 0$, let $H_{k}\left(\Delta_{\mathcal{F},\omega}^{u,v}\right)$ be the Hilbert space completion of the space $\Omega^{u,v}\left(M,\mathcal{F}\right)$ with respect to the scalar product $$\left\langle\alpha,\beta\right\rangle=\sum_{j=0}^{j=k}{\left\langle\left(\Delta_{\mathcal{F},\omega}^{u,v}\right)^{j}\alpha,\beta\right\rangle}$$ for $\alpha,\beta\in\Omega^{u,v}\left(M,\mathcal{F}\right)$. For the corresponding norm $\left\|.\right\|_{k}$, we have 
$$k\leq k^{\prime}\Rightarrow\left\|\alpha\right\|_{k}\leq\left\|\alpha\right\|_{k^{\prime}}\text{   for all   }\alpha\in\Omega^{u,v}\left(M,\mathcal{F}\right).$$
Thus we obtain the chain of continuous inclusions 
$$H=H_{0}\left(\Delta_{\mathcal{F},\omega}^{u,v}\right)\supset H_{1}\left(\Delta_{\mathcal{F},\omega}^{u,v}\right)\supset H_{2}\left(\Delta_{\mathcal{F},\omega}^{u,v}\right)\supset\cdots\supset H_{\infty}\left(\Delta_{\mathcal{F},\omega}^{u,v}\right),$$  where $$H_{\infty}\left(\Delta_{\mathcal{F},\omega}^{u,v}\right)=\bigcap_{k\geq 0}H_{k}\left(\Delta_{\mathcal{F},\omega}^{u,v}\right)$$ equipped with the Frech{\'e}t topology.

\begin{theorem}
\label{ta25}
Let $\left(M,\mathcal{F}\right)$ be a smooth foliation of a closed Riemannian manifold with a bundle like metric and $\omega\in\Omega^{0,1}\left(M,\mathcal{F}\right)$. The Laplacian $\Delta_{\mathcal{F},\omega}^{u,v}$ on $\Omega^{u,v}\left(M,\mathcal{F}\right)$ gives rise to an orthogonal direct sum decomposition 

$$H_{\infty}\left(\Delta_{\mathcal{F},\omega}^{u,v}\right)\cong\ker\overline{\Delta_{\mathcal{F},\omega,\infty}^{u,v}}\oplus\overline{\mathrm{im}\overline{\Delta_{\mathcal{F},\omega,\infty}^{u,v}}}\cong\ker\overline{\Delta_{\mathcal{F},\omega,\infty}^{u,v}}\oplus\overline{\mathrm{im}\overline{P_{\infty}}}\oplus\overline{\mathrm{im}\overline{P^{*}_{\infty}}},$$

 where $\overline{\Delta_{\mathcal{F},\omega,\infty}^{u,v}}$, $\overline{P_{\infty}}$, and $\overline{P^{*}_{\infty}}$ are canonical continuous extensions of the corresponding differential operators.
\end{theorem}
\begin{proof}
The complexification of the Dirac operator $D=P+P^{*}$ satisfies the hypothesis of Chernoff's Lemma 2.1 in \cite{chernoff1973essential}. This can be verified from Corollary 1.4 of \cite{chernoff1973essential}. Then with the ideas explained in Section 2 of \cite{lopez1991hodge}, we have the real Hilbert spaces $H_{k}\left(\Delta_{\mathcal{F},\omega}^{u,v}\right)$ and $H_{\infty}\left(\Delta_{\mathcal{F},\omega}^{u,v}\right)$. We can extend the operator $D$ to $\overline{D_{\infty}}$ and $\Delta_{\mathcal{F},\infty}^{\omega}$ to

$$\overline{\Delta_{\mathcal{F},\omega,\infty}^{u,v}}: H_{\infty}\left(\Delta_{\mathcal{F},\omega}^{u,v}\right)\rightarrow H_{\infty}\left(\Delta_{\mathcal{F},\omega}^{u,v}\right),$$
yielding the orthogonal decompositions 

$$H_{\infty}\left(\Delta_{\mathcal{F},\omega}^{u,v}\right)\cong\ker\overline{\Delta_{\mathcal{F},\omega,\infty}^{u,v}}\oplus\overline{\mathrm{im}\overline{\Delta_{\mathcal{F},\omega,\infty}^{u,v}}}\cong\ker\overline{D_{\infty}}\oplus\overline{\mathrm{im}\overline{D_{\infty}}}.$$

Notice the spaces $\Omega^{p,q}\left(M,\mathcal{F}\right)$ are orthogonal to each other with respect to the inner product $\left\langle,\right\rangle_{k}$ defined above, for eack $k\geq 0$. Therefore, it follows that $\overline{D_{\infty}}$ can be decomposed as the sum of the continuous operators $$\overline{P_{\infty}},\overline{P_{\infty}^{*}}:H_{\infty}\rightarrow H_{\infty},$$ which are extensions of $P_{\infty}$ and $P_{\infty}^{*}$ respectively. Since $\mathrm{im} P$ and $\mathrm{im} P^{*}$ are $\left\langle,\right\rangle_{k}$-orthogonal for each $k\geq 0$, we obtain the following orthogonal decomposition: $$H_{\infty}\left(\Delta_{\mathcal{F},\omega}^{u,v}\right)\cong\ker\overline{\Delta_{\mathcal{F},\omega,\infty}^{u,v}}\oplus\overline{\mathrm{im}\overline{P_{\infty}}}\oplus\overline{\mathrm{im}\overline{P_{\infty}^{*}}}.$$
\end{proof}

\begin{corollary}
\label{ta26}
Every reduced general leafwise Morse-Novikov cohomology class has a $\Delta_{\mathcal{F},\omega}^{u,v}$-harmonic representative.
\end{corollary}
For any $\alpha\wedge\beta\in\Omega^{u,v}\left(M,\mathcal{F}\right)$, from formula 17 in \cite{lopez1991hodge} we have $\delta_{\mathcal{F}}=(-1)^{n(u+v)+n+1}*d_{\mathcal{F}}*$. By using the identities 
\begin{eqnarray*} 
\omega\lrcorner=(-1)^{n(u+v)+n}*\omega\wedge*\\
\text{ and }    *^{2}=(-1)^{(u+v)(n+1)},
\end{eqnarray*}
it can be shown that
\begin{eqnarray*}
\left(\omega\lrcorner\right)*=\left(-1\right)^{u+v}*\left(\omega\wedge\right).\\
*\left(\omega\lrcorner\right)=\left(-1\right)^{u+v+1}\left(\omega\wedge\right)*.\\
*\delta_{\mathcal{F}}=\left(-1\right)^{u+v}d_{\mathcal{F}}*.\\
\delta_{\mathcal{F}}*=\left(-1\right)^{u+v+1}*d_{\mathcal{F}}.
\end{eqnarray*}

\begin{proposition}
\label{ta27}
Let $\left(M,\mathcal{F}\right)$, and $\omega$ be as in Theorem \ref{ta25}. The Hodge operator $*:\Omega^{u,v}\left(M,\mathcal{F}\right)\rightarrow\Omega^{q-u,p-v}\left(M,\mathcal{F}\right)$ satisfies $$*\Delta_{\mathcal{F},\omega}^{u,v}=\Delta_{\mathcal{F},-\omega}^{q-u,p-v}*.$$
\end{proposition} 
\begin{proof}
Similar to Proposition \ref{ta19}, using the formulas above.
Thus the operator $*$ maps $\Delta_{\mathcal{F},\omega}^{u,v}$-harmonic forms to $\Delta_{\mathcal{F},-\omega}^{q-u,p-v}$-harmonic forms.
\end{proof}

\begin{lemma}
\label{ta28}
$\omega\lrcorner=(-1)^{p\left(v+1\right)}*_{\mathcal{F}}\left(\omega\wedge\right)*_{\mathcal{F}}$ for all $\omega\in\Omega^{0,1}\left(M,\mathcal{F}\right)$ on $\Omega^{u,v}\left(M,\mathcal{F}\right)$.
\end{lemma}
\begin{proof}
For any $\alpha\wedge\beta\in\Omega^{u,v}\left(M,\mathcal{F}\right)$, we have
\begin{eqnarray*}
\omega\lrcorner\left(\alpha\wedge\beta\right)&=&\left(-1\right)^{n\left(u+v+1\right)}*\left(\omega\wedge\right)*\left(\alpha\wedge\beta\right)\\
&=&\left(-1\right)^{n\left(u+v+1\right)}\left(-1\right)^{\left(q-u\right)v}\left(-1\right)^{\left(q-u\right)}*\left(*_{\bot}\alpha\wedge\omega\wedge*_{\mathcal{F}}\beta\right)\\
&=&\left(-1\right)^{n\left(u+v+1\right)}\left(-1\right)^{\left(q-u\right)v}\left(-1\right)^{\left(q-u\right)}\left(-1\right)^{\left(q-q+u\right)\left(p-v+1\right)} *_{\bot}^{2}\alpha\wedge *_{\mathcal{F}}\left(\omega\wedge*_{\mathcal{F}}\beta\right)\\
&=&\left(-1\right)^{n\left(u+v+1\right)}\left(-1\right)^{\left(q-u\right)v}\left(-1\right)^{\left(q-u\right)}\left(-1\right)^{\left(u\right)\left(p-v+1\right)} \left(-1\right)^{u\left(q+1\right)}\alpha\wedge *_{\mathcal{F}}\left(\omega\wedge*_{\mathcal{F}}\beta\right)\\
&=&\left(-1\right)^{n\left(u+v+1\right)}\left(-1\right)^{\left(q-u\right)v}\left(-1\right)^{\left(q-u\right)}\left(-1\right)^{\left(u\right)\left(p-v+1\right)} \left(-1\right)^{u\left(q+1\right)}\left(-1\right)^{u\left(v-1\right)}*_{\mathcal{F}}\left(\omega\wedge\right)*_{\mathcal{F}}\beta\wedge\alpha\\
&=&\left(-1\right)^{p\left(v+1\right)+uv}\left(-1\right)^{uv}*_{\mathcal{F}}\left(\omega\wedge\right)*_{\mathcal{F}}\alpha\wedge\beta\\
&=&\left(-1\right)^{p\left(v+1\right)}*_{\mathcal{F}}\left(\omega\wedge\right)*_{\mathcal{F}}\alpha\wedge\beta.
\end{eqnarray*}
\end{proof}

In a Riemannian foliation, any element of $\Omega^{u,v}(M,\mathcal{F})$, on a local foliated chart, can be expressed as a linear combination of forms of the type 
$\alpha\wedge\beta$, where $\alpha$ is a basic form in $\Omega^{u,0}(M,\mathcal{F})$ and $\beta$ is a form in $\Omega^{0,v}(M,\mathcal{F})$. Then in particular $d_{\mathcal{F}}\alpha =0$. See \cite[Lemma 3.4]{lopez2001long}.

\begin{lemma}
\label{ta29}
$d_{\mathcal{F}}^{*}=\left(-1\right)^{p\left(v+1\right)+1}*_{\mathcal{F}} d_{\mathcal{F}} *_{\mathcal{F}}$ for all on $\Omega^{u,v}\left(M,\mathcal{F}\right)$.
\end{lemma}
\begin{proof}
For any $\alpha\wedge\beta\in\Omega^{u,v}\left(M,\mathcal{F}\right)$, we have
\begin{eqnarray*}
d_{\mathcal{F}}^{*}\left(\alpha\wedge\beta\right) &=& \left(-1\right)^{n\left(u+v\right)+n+1}*d_{\mathcal{F}}*\left(\alpha\wedge\beta\right)\\
&=& \left(-1\right)^{n\left(u+v\right)+n+1+\left(q-u\right)\left(v+1\right)} * \left(*_{\perp}\alpha\wedge d_{\mathcal{F}} *_{\mathcal{F}}\beta\right)\\
&=& \left(-1\right)^{n\left(u+v\right)+n+1+\left(q-u\right)\left(v+1\right)+u\left(p-v+1\right)+u(q+1)}\left(\alpha\wedge\left(*_{\mathcal{F}} d_{\mathcal{F}} *_{\mathcal{F}}\beta\right)\right)\\
&=& \left(-1\right)^{p\left(v+1\right)+u+1}\left(\alpha\wedge\left(*_{\mathcal{F}} d_{\mathcal{F}} *_{\mathcal{F}}\beta\right)\right)\\
&=& \left(-1\right)^{p\left(v+1\right)+u+1+u\left(v-1\right)}\left(*_{\mathcal{F}} d_{\mathcal{F}} *_{\mathcal{F}}\beta\right)\wedge\alpha\\
&=& \left(-1\right)^{p\left(v+1\right)+1}\left(*_{\mathcal{F}} d_{\mathcal{F}} *_{\mathcal{F}}\right)\alpha\wedge\beta.
\end{eqnarray*}
\end{proof}
By Lemmas above and the identity $*_{\mathcal{F}}^{2}=\left(-1\right)^{v\left(p-v\right)}$, we have on $\Omega^{u,v}\left(M,\mathcal{F}\right)$
\begin{eqnarray*}
\left(\omega\lrcorner\right)*_{\mathcal{F}}&=&\left(-1\right)^{p(p-v+1)}\left(\omega\lrcorner\right)*_{\mathcal{F}}^{2}\\
&=&\left(-1\right)^{P(p-v+1)+v(p-v)}*_{\mathcal{F}}\left(\omega\wedge\right)\\
&=&\left(-1\right)^{v}*_{\mathcal{F}}\left(\omega\wedge\right), 
\end{eqnarray*}
and
\begin{eqnarray*}
*_{\mathcal{F}}\left(\omega\lrcorner\right) &=& \left(-1\right)^{p\left(v+1\right)}*_{\mathcal{F}}^{2}\left(\omega\wedge\right)*_{\mathcal{F}}\\
&=& \left(-1\right)^{p\left(v+1\right)}\left(-1\right)^{\left(p-v\right)\left(v+1\right)}\left(\omega\wedge\right)*_{\mathcal{F}}\\
&=& \left(-1\right)^{\left(v+1\right)}\left(\omega\wedge\right)*_{\mathcal{F}}.
\end{eqnarray*} 

Similarly on $\Omega^{u,v}\left(M,\mathcal{F}\right)$ we also have

\begin{eqnarray*}
*_{\mathcal{F}}d_{\mathcal{F}}^{*}&=&\left(-1\right)^{v}d_{\mathcal{F}}*_{\mathcal{F}}\\
d_{\mathcal{F}}^{*}*_{\mathcal{F}}&=&\left(-1\right)^{v+1}*_{\mathcal{F}}d_{\mathcal{F}}.
\end{eqnarray*}

\begin{proposition}
\label{ta30}
Let $\left(M,\mathcal{F}\right)$, and $\omega$ be as in Theorem \ref{ta25}. The Hodge operator $*_{\mathcal{F}}:\Omega^{u,v}\left(M,\mathcal{F}\right)\rightarrow\Omega^{u,p-v}\left(M,\mathcal{F}\right)$ satisfies $$*_{\mathcal{F}}\Delta_{\mathcal{F},\omega}^{u,v}=\Delta_{\mathcal{F},-\omega}^{u,p-v}*_{\mathcal{F}}.$$
\end{proposition} 
\begin{proof}
Similar to Proposition \ref{ta19}, using the formulas above.
Thus the operator $*_{\mathcal{F}}$ maps $\Delta_{\mathcal{F},\omega}^{u,v}$-harmonic forms to $\Delta_{\mathcal{F},-\omega}^{u,p-v}$-harmonic forms.
\end{proof}
Notice on $\Omega^{u,v}\left(M,\mathcal{F}\right)$, we have
\begin{eqnarray*}
**_{\mathcal{F}}=(-1)^{(q-u)(p-v)}*_\perp *_{\mathcal{F}}*_{\mathcal{F}}\\
**_{\mathcal{F}}=(-1)^{(q-u)(p-v)+v(p-v)}*_\perp\\
*_\perp=(-1)^{(q-u)(p-v)+v(p-v)}**_{\mathcal{F}}.
\end{eqnarray*} 
From the Propositions \ref{ta27}and \ref{ta29} we have,

\begin{corollary}
\label{ta31}
Let $\left(M,\mathcal{F}\right)$, and $\omega$ be as in Theorem \ref{ta25}. The Hodge operator $*_{\perp}:\Omega^{u,v}\left(M,\mathcal{F}\right)\rightarrow\Omega^{q-u,v}\left(M,\mathcal{F}\right)$ satisfies $$*_{\perp}\Delta_{\mathcal{F},\omega}^{u,v}=\pm\Delta_{\mathcal{F},\omega}^{q-u,v}*_{\perp}.$$
\end{corollary}

\begin{proof}
\begin{eqnarray*}
*_{\perp}\Delta_{\mathcal{F},\omega}^{u,v}&=&\left(-1\right)^{v(n-u-v)} *  *_{\mathcal{F}}\Delta_{\mathcal{F},\omega}^{u,v}\\
&=& \left(-1\right)^{v(n-u-v)} *  \Delta_{\mathcal{F},-\omega}^{u,p-v}*_{\mathcal{F}}\\
&=& \left(-1\right)^{v(n-u-v)} \Delta_{\mathcal{F},\omega}^{q-u,v}* *_{\mathcal{F}}\\
&=& \left(-1\right)^{v(n-u-v)}\left(-1\right)^{(q-u)(p-v)+v(p-v)}\Delta_{\mathcal{F},\omega}^{q-u,v} *_{\perp}\\
&=& \left(-1\right)^{p(q-u)}\Delta_{\mathcal{F},\omega}^{q-u,v} *_{\perp}.
\end{eqnarray*}
\end{proof}

Let $h^{u,v}_{\omega}$, $h^{u,v}_{-\omega}$ be the dimensions of the Cohomology groups $H_{\omega}^{u,v}\left(M,d_{\mathcal{F}}\right)$
, and $H_{-\omega}^{u,v}\left(M,d_{\mathcal{F}}\right)$ respectively. We observe that,
Proposition \ref{ta27} implies, $h^{u,v}_{\omega}$=$h^{q-u,p-v}_{-\omega}$, Proposition \ref{ta30} implies, $h^{u,v}_{\omega}$=$h^{u,p-v}_{-\omega}$, and Corollary \ref{ta31} implies, $h^{u,v}_{\omega}$=$h^{q-u,v}_{\omega}$.

\begin{corollary}
\label{hodgediamond}
( Hodge Diamond Structure )For a Riemannian foliation of a manifold 
$$h^{u,v}_{\omega}=h^{q-u,p-v}_{-\omega}=h^{u,p-v}_{-\omega}=h^{q-u,v}_{\omega}.$$
\end{corollary}

In particular, we consider a diagram of the dimensions of the cohomology classes for a manifold of dimension $n=5$ with $p=2$, and $q=3$.\\\\
\hspace{-2.8in}\includegraphics[scale=1]{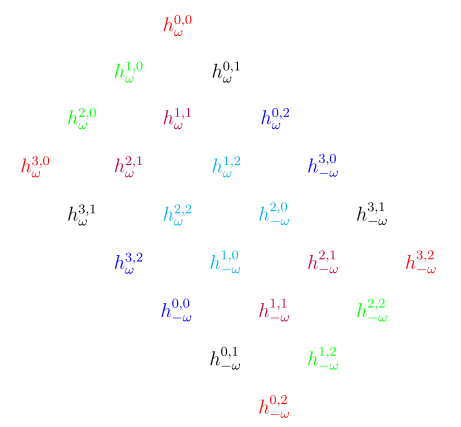}

\begin{comment}

	\begin{array}{rrrrrrrrrrrrrrr}
	& & & & & & & & &  \color{red}{h^{0,0}_{\omega}} & & & & & & & & &\\ 
    & & & & & & & & \color{green}{h^{1,0}_{\omega}} & & h^{0,1}_{\omega} & & & & & & & &\\
	& & & & & & &  \color{green}{h^{2,0}_{\omega}} & & \color{purple}{h^{1,1}_{\omega}} & &  \color{blue}{h^{0,2}_{\omega}} & & & & & & &\\
	& & & & & &  \color{red}{h^{3,0}_{\omega}} & & \color{purple}{h^{2,1}_{\omega}} & & \color{yellow}{h^{1,2}_{\omega}} & &  \color{blue}{h^{3,0}_{-\omega}} & & & & & &\\
	& & & & & & &  h^{3,1}_{\omega} & & \color{yellow}{h^{2,2}_{\omega}} & & \color{yellow}{h^{2,0}_{-\omega}} & &  h^{3,1}_{-\omega} & & & & & &\\
	& & & & & & & &  \color{blue}{h^{3,2}_{\omega}} & & \color{yellow}{h^{1,0}_{-\omega}} & & \color{purple}{h^{2,1}_{-\omega}} & &  \color{red}{h^{3,2}_{-\omega}}  & & & & &\\
	& & & & & & & & &  \color{blue}{h^{0,0}_{-\omega}} & & \color{purple}{h^{1,1}_{-\omega}} & & \color{green}{h^{2,2}_{-\omega}} & & & & & & & & & & & &\\
	& & & & & & & & & &  h^{0,1}_{-\omega} & & \color{green}{h^{1,2}_{-\omega}} & & & & & & & & & & & & & &\\
	& & & & & & & & & & & \color{red}{h^{0,2}_{-\omega}} & & & & & & & & & &
    \end{array}

\end{comment}
The dimensions in the same color are equal. 

\newpage
\bibliographystyle{amsplain}
\bibliography{NewbibliographySharif}
\end{document}